\pdfoutput=1
\documentclass[a4paper]{article}
\usepackage[T2A]{fontenc}
\usepackage[cp1251]{inputenc}
\usepackage[russian,english]{babel}
\usepackage[tbtags]{amsmath}
\usepackage{amsfonts,amssymb,mathrsfs,amscd}
 \usepackage{msb-a}
 
\newlength\tindent
\renewcommand{\indent}{\hspace*{\tindent}}
\JournalName{}
\numberwithin{equation}{section}
\theoremstyle{plain}

\newtheorem{theorem}{Theorem}
\theoremstyle{plain}

\newtheorem{lemma}{Lemma}
\newtheorem{corollary}{Corollary}
\theoremstyle{definition}
\newtheorem{proof}{}

\newtheorem{remark}{Remark}
\renewcommand{\refname}{References}
\renewcommand{\abstract}{\begin{bf} Abstract.\end{bf}}
\begin{document}

\title{Nov\'ak-Carmichael numbers and shifted primes without large prime factors}
\author{Alexander Kalmynin}
\address{National Research University Higher School of Economics, Russian Federation, Math Department, International Laboratory of Mirror Symmetry and Automorphic Forms}
\email{alkalb1995cd@mail.ru}
\date{}
\udk{}
\maketitle
\markright{Nov\'ak-Carmichael numbers}
\begin{abstract}
We prove some new lower bounds for the counting function $\mathcal N_{\mathcal C}(x)$ of the set of Nov{\'a}k-Carmichael numbers. Our estimates depend on the bounds for the number of shifted primes without large prime factors. In particular, we prove that $\mathcal N_{\mathcal C}(x) \gg x^{0.7039-o(1)}$ unconditionally and that $\mathcal N_{\mathcal C}(x) \gg xe^{-(7+o(1))(\log x)\frac{\log\log\log x}{\log\log x}}$, under some reasonable hypothesis.
\end{abstract}
\begin{fulltext}
\footnotetext[0]{The author is partially supported by Laboratory of Mirror Symmetry NRU HSE, RF Government grant, ag. № 14.641.31.0001, the Simons Foundation and the Moebius Contest Foundation for Young Scientists}

\section{Introduction}

In the paper \cite{Kalm}, author introduced the Nov\'ak-Carmichael numbers. Positive integer $N$ is called a Nov\'ak-Carmichael number if for any $a$ coprime to $N$ the congruence $a^N \equiv 1 \pmod N$ holds. Later, S.\,V.\,Konyagin posed a problem about the order of growth of the quantity $\mathcal N_{\mathcal C}(x)$ ---  the number of Nov\'ak-Carmichael numbers which are less than or equal to $x$. The present work provides a partial answer to this question.

It turns out that the lower bounds for the quantity $\mathcal N_{\mathcal C}(x)$ can be deduced from the theorems on the distribution of shifted prime numbers without large prime factors. Namely, for a positive integers $x$ and $y$ denote by $\mathcal P(x,y)$ the set of all prime numbers $p \leq x$ such that the largest prime factor of $p-1$ is less than or equal to $y$. Let also $\Pi(x,y)$ be the number of elements of the set $\mathcal P(x,y)$. Then the following proposition holds:

\begin{theorem}

Let $u$ be some fixed real number with $0<u<1$. If for $z\to +\infty$ we have

$$\Pi(z,z^u)=z^{1+o(1)},$$

then the lower bound

$$\mathcal N_{\mathcal C}(x) \gg x^{1-u+o(1)}$$

holds.
\end{theorem}
Lower bounds for the quantity $\Pi(z,z^u)$ for different values of $u$ are studied in the papers \cite{Pom},\cite{Fried},\cite{BakHarm}. In particular, using the result of the last article we obtain

\begin{corollary}
The inequality

$$\mathcal N_{\mathcal C}(x) \gg x^{1-\beta+o(1)},$$

is true for $\beta=0.2961$.
\end{corollary}
\begin{remark}
It is conjectured that for any fixed positive $u$ we have  $$\Pi(z,z^u)\asymp_u \pi(z) \sim \frac{z}{\log z}.$$ It is also reasonable to conjecture that for any nice enough function $y(z)$ the asymptotic relation
\begin{equation}
\label{optimistic}
\frac{\Pi(z,y(z))}{\pi(z)} \sim \frac{\Psi(z,y(z))}{z}
\end{equation}

holds, where $\Psi(z,y)$ is the number of natural numbers $n \leq z$ such that the largest prime factor of $n$ is less than or equal to $y$.
\end{remark}

 For example, if we assume that relation (\ref{optimistic}) is true for $y(z)=e^{\sqrt{\log z}}$, then by the formula

$$\Psi(z,e^{\sqrt{\log z}})=ze^{-(1/2+o(1))\sqrt{\log z}\log\log z}$$

(see \cite{Hild}) we get

$$\Pi(z,e^{\sqrt{\log z}}) \gg ze^{-(1/2+o(1))\sqrt{\log z}\log\log z}.$$

Using a slightly weaker form of this assumption, we improve the estimate of the Theorem 1:

\begin{theorem}
Suppose that for some fixed constant $c>0$ the inequality

$$\Pi(z,e^{\sqrt{\log z}}) \gg ze^{-c\sqrt{\log x}\log\log x}$$

holds. Then for any $d>6+2c$ we have

$$\mathcal N_{\mathcal C}(x) \gg xe^{-d\log x\frac{\log\log\log x}{\log\log x}}.$$

In particular, if relation (\ref{optimistic}) is true for $y(z)=e^{\sqrt{\log z}}$, then

$$\mathcal N_{\mathcal C}(x) \gg xe^{-(7+o(1))\log x\frac{\log\log\log x}{\log\log x}}.$$
\end{theorem}

\section{Proofs of the theorems}

The constructions that we will use in our proofs are largely similar to that of papers \cite{Pom}, \cite{AGLPS}. First of all, we need a description of Nov\'ak-Carmichael numbers in terms of their prime factors, which is an analogue of Koselt's criterion (cf. \cite{Kors}) for Carmichael numbers:

\begin{lemma}
Natural number $n$ is a Nov\'ak-Carmichael number if and only if for any prime divisor $p$ of $n$ the number $p-1$ also divides $n$.
\end{lemma}

\begin{proof}[Proof]
Let $n=2^\alpha\prod_{k=1}^{m} p_k^{\alpha_k}$, where $p_k$ are distinct odd prime numbers and $\alpha_k \geq 1$. If $n$ is a Nov\'ak-Carmichael number, then for any $a$ coprime to $n$ and any $k$ we have

$$a^n \equiv 1 \pmod {p_k^{\alpha_k}}.$$

On the other hand, by the Chinese remainder theorem we can choose $a$ such that for any $k$ the congruence

$$a \equiv g_k \pmod {p_k^{\alpha_k}}$$

holds, where $g_k$ is some primitive root modulo $p_k^{\alpha_k}$.

Consequently, for any $k$ we have

$$g_k^n \equiv a^n\equiv 1 \pmod {p_k^{\alpha_k}}.$$

Thus, for any $k$ the number $n$ is divisible by the multiplicative order of $g_k$ modulo $p_k^{\alpha_k}$. Hence $p_k^{\alpha_k-1}(p_k-1)$ divides $n$. So, for any odd prime divisor $p$ of $n$ we have $p-1 \mid n$. Also, $2-1$ divides $n$.

Conversely, if for any prime $p$ dividing $n$ the number $p-1$ also divides $n$, then for any $k$ we have $\varphi(p_k^{\alpha_k})\mid n$ and $\varphi(2^\alpha)\mid n$. Hence, if $(a,n)=1$, then

$$a^n=(a^{\varphi(p_k^{\alpha_k})})^{n/\varphi(p_k^{\alpha_k})} \equiv 1 \pmod {p_k^{\alpha_k}}$$

for any $k$ and

$$a^n=(a^{\varphi(2^\alpha)})^{n/\varphi(2^\alpha)} \equiv 1 \pmod {2^\alpha}.$$

From these congruences and pairwise coprimality of numbers $2^\alpha,p_1^{\alpha_1},\ldots,p_m^{\alpha_m}$ we obtain

$$a^n\equiv 1 \pmod n,$$

as needed.
\end{proof}

For the asymptotic estimates of sizes of certain sets the following inequality involving binomial coefficients is needed:

\begin{lemma}
Let $a$ and $b$ be a positive integers with $b\leq a/2+1$. Then we have

$${a\choose b} \geq \left(\frac{a}{b}\right)^b.$$
\end{lemma}

\begin{proof}[Proof]
Let us prove this statement by induction over $b$.

The case $b=1$ is obvious, since

$${a \choose b}=a=\left(\frac{a}{1}\right)^1.$$

Suppose now that $1<c+1 \leq a/2+1$ and the inequality is true for $b=c$. Then we have

$${a \choose {c+1}}={a \choose c}\frac{a-c}{c+1}\geq \left(\frac{a}{c}\right)^c\frac{a-c}{c+1}=\left(\frac{a}{c+1}\right)^{c+1}\left(1-\frac{c}{a}\right)\left(1+\frac{1}{c}\right)^c.$$

On the other hand, $c\leq a/2$ and $(1+\frac{1}{c})^c \geq 2$, so the inequality

$${a \choose {c+1}}\geq \left(\frac{a}{c+1}\right)^{c+1}$$

holds, which was to be proved.
\end{proof}

In the next lemma, for arbitrary real numbers $r$ and $s$ satisfying the inequality $2\leq r\leq s$ we will construct the number $\mathcal D(r,s)$ with some remarkable properties.

\begin{lemma}
Let $s,r \in \mathbb R$ and $2 \leq r \leq s$. If

$$\mathcal D(s,r)=\prod_{p \leq r} p^{\left[\frac{\log s}{\log p}\right]},$$

where the product is taken over prime numbers $p$, then

$$\log \mathcal D(s,r)=O\left(\frac{r\log s}{\log r}\right)$$

and for any subset $\mathcal A \subseteq \mathcal P(s,r)$ the number

$$\mathcal E(\mathcal A,s,r)=\mathcal D(s,r)\prod_{p \in \mathcal A} p$$

is a Nov\'ak-Carmichael number.

\end{lemma}

\begin{proof}[Proof]
Indeed,

$$\log\mathcal D(s,r)=\sum_{p \leq r} \left[\frac{\log s}{\log p}\right]\log p\leq \sum_{p \leq r} (\log p) \frac{\log s}{\log p}=\pi(r)\log s=O\left(\frac{r\log s}{\log r}\right).$$

Let us prove now that the number $\mathcal E(\mathcal A,s,r)$ is a Nov\'ak-Carmichael number. Suppose that $q$ is a prime factor of $\mathcal E(\mathcal A,s,r)$. Then we have either $q\mid \mathcal D(s,r)$ or $q \in \mathcal P(s,r)$. But all the prime factors of $\mathcal D(s,r)$ are not exceeding $r$ and so are lying in $\mathcal P(s,r)$. Thus, $q \in \mathcal P(s,r)$. 

Consequently, $q-1=\prod_{l=1}^{t} p_l^{\beta_l}$ and $p_l \leq r$ for any $l$. On the other hand, $p_l^{\beta_l} \leq q-1<s$. Taking the logarithms, we obtain $\beta_l \leq \left[\frac{\log s}{\log p_l}\right]$. Thus, for any $l$ we have $p_l^{\beta_l} \mid \mathcal D(s,r)$, so $q-1 \mid \mathcal D(s,r) \mid \mathcal E(\mathcal A,s,r)$. Hence, by the Lemma 1, our number is a Nov\'ak-Carmichael number. This concludes the proof.
\end{proof}

Let us now prove Theorems 1 and 2.

\begin{proof}[Proof of Theorem 1]

Suppose that $0<u<1$ and $\Pi(z,z^u)=z^{1+o(1)}$ as $z \to \infty$. We introduce the notation $$r=\frac{\log x}{\log\log^2 x},\, s=r^{1/u}$$ and $$A=\left[u\frac{\log x}{\log r}-u\frac{\log \mathcal D(s,r)}{\log r}\right].$$ By the Lemma 3 we have $$\log\mathcal D(s,r)=O\left(\frac{r\log s}{\log r}\right)=O\left(\frac{\log x}{\log\log x}\right)=o(\log x)$$ hence, $A=(u+o(1))\frac{\log x}{\log r}$. Now, for any subset $\mathcal A \subseteq \mathcal P(s,r)$ of cardinality $A$ consider the number $\mathcal E(\mathcal A,s,r)$. By the Lemma 3 this number is a Nov\'ak-Carmichael number and

$$\mathcal E(\mathcal A,s,r)=\mathcal D(s,r)\prod_{p \in \mathcal A} p \leq \mathcal D(s,r)s^A=e^{\log\mathcal D(s,r)+A\log s}.$$

Note that $A=\left[u\frac{\log x}{\log r}-u\frac{\log \mathcal D(s,r)}{\log r}\right]\leq u\frac{\log x}{\log r}-u\frac{\log \mathcal D(s,r)}{\log r}=\frac{\log x}{\log s}-\frac{\log\mathcal D(s,r)}{\log s}$. From this we obtain the inequality

$$\log\mathcal E(\mathcal A,s,r)\leq \log\mathcal D(s,r)+A\log s\leq \log x.$$

Hence, all the constructed numbers $\mathcal E(\mathcal A,s,r)$ are less than or equal to $x$. Furthermore, all these numbers are distinct, as otherwise for some different subsets $\mathcal A, \mathcal B$ we would have had

$$\mathcal E(\mathcal A,s,r)=\mathcal D(s,r)\prod_{p \in \mathcal A}p=\mathcal D(s,r)\prod_{p \in \mathcal B} p=\mathcal E(\mathcal B,s,r)$$

hence, $\prod\limits_{p \in \mathcal A} p=\prod\limits_{p \in \mathcal B} p$, which is not the case.

So, the number of Nov\'ak-Carmichael numbers not exceeding $x$ is at least as large as the number of subsets in $\mathcal P(s,r)$ of cardinality $A$. But for large enough $x$ we have $$A\ll \log x<(\log x)^{1/u+o(1)}=\Pi(s,r)/2.$$ Consequently, using Lemma 2 we get

$$\mathcal N_{\mathcal C}(x)\geq {\Pi(s,r) \choose A}\geq \left(\frac{\Pi(s,r)}{A}\right)^A.$$

From $$\Pi(s,r)=s^{1+o(1)}=(\log x)^{1/u+o(1)}$$ and $$A=(u+o(1))\frac{\log x}{\log r}=(u+o(1))\frac{\log x}{\log\log x}$$ we finally get

$$\mathcal N_{\mathcal C}(x) \geq (\log x)^{(1/u-1+o(1))A}=e^{(1/u-1+o(1))(u+o(1))\frac{\log x}{\log\log x}\log\log x}=x^{1-u+o(1)},$$

which is the required result.
\end{proof}

The proof of Theorem 2 is proceeded analogously. All we need is some different choice of parameters $r,s$ and $A$.

\begin{proof}[Proof of Theorem 2]

Assume that $\Pi(z,e^{\sqrt{\log z}})\gg ze^{-c\sqrt{\log z}\log\log z}$. Let us choose $$r=\frac{\log x}{(\log\log x)^3},\, s=e^{(\log\log x-3\log\log\log x)^2}=e^{\log^2 r}$$ and, as before, $$A=\left[\frac{\log x}{\log s}-\frac{\log\mathcal D(s,r)}{\log s}\right].$$ Now, similarly to the proof of Theorem 1, considering the subsets of $\mathcal P(s,r)$ which contain exactly $A$ elements we obtain

$$\mathcal N_{\mathcal C}(x)\geq \left(\frac{\Pi(s,r)}{A}\right)^A.$$

Furthermore, by Lemma 3 we have $A=\frac{\log x}{\log s}+O\left(\frac{r\log s}{\log r}\right)\geq\frac{\log x}{(\log\log x)^2}$. Also, due to the assumption of the theorem, we have $\Pi(s,r)\gg se^{-c\sqrt{\log s}\log\log s}\geq se^{-2c\log x\log\log x}$. So, for any $d>6+2c$ the inequality

$$\frac{\Pi(s,r)}{A} \gg e^{(\log\log x)^2-d(\log\log x)\log\log\log x}$$

holds.

Thus, we have

$$\mathcal N_{\mathcal C}(x) \gg e^{A(\log\log x)^2-dA(\log\log x)\log\log\log x}\geq e^{\log x-d(\log x)\frac{\log\log\log x}{\log\log x}}=xe^{-d(\log x)\frac{\log\log\log x}{\log\log x}},$$

which concludes the proof of Theorem 2.
\end{proof}

\section{Conclusion}

We showed that lower bounds for the number of shifted prime numbers without large prime factors imply some nice lower bounds for the counting function of the set of Nov\'ak-Carmichael numbers. It is a well-known fact that these theorems also provide estimates for the counting function of Carmichael numbers (cf. \cite{AGP}). However, in our situation it is possible to use much simplier constructions.

Furthermore, the relation (\ref{optimistic}) for $y(z)=e^{\sqrt{\log z}}$ implies the lower bound which is as strong as the upper bound for the number of Carmichael numbers less than a given magnitude proved by P.\,Erd{\"o}s. Unfortunately, the methods of the paper \cite{Erdos} do not allow a direct generalization to the case of Nov\'ak-Carmichael numbers. So, the problem of obtaining the correct order of growth of the quantity $\mathcal N_{\mathcal C}(x)$ remains open even on the assumption of the relation (\ref{optimistic}).
\end{fulltext}


\begin{thebibliography}{99}
\bibitem{AGLPS} J.\,J.\,Alba Gonzalez, F.\,Luca, C.\,Pomerance, I.\,E.\,Shparlinski, <<On numbers $n$ dividing the $n$th term of a linear recurrence>>, Proc. Edinburgh Math. Soc., 55 (2012), 271-289.
\bibitem{AGP} W.\,R.\,Alford, A.\,Granville, C.\,Pomerance, <<There are infinitely many Carmichael numbers>>, Ann. of Math. (2) 139 (1994), 703-722.
\bibitem{BakHarm} R.\,C.\,Baker, G.\,Harman. <<Shifted primes without large prime factors.>> Acta Arithmetica 83:4 (1998), 331-361.
\bibitem{Erdos} P.\,Erd{\"o}s, <<On pseudoprimes and Carmichael numbers>>, Publ. Math. Debrecen 4 (1956), 201-206.
\bibitem{Fried} J.\,Friedlander, <<Shifted primes without large prime factors>>, Number Theory and Applications, (1989), Kluwer, Berlin, 393-401.
\bibitem{Hild} A.\,Hildebrand, <<On the number of positive integers $\leq x$ and free of prime factors $>y$>>, J. Number Theory 22:3 (1986), 289-307.
\bibitem{Kalm} A.\,B.\,Kalmynin, <<On Nov\'ak numbers>>, arXiv:1611.00417 (2016).
\bibitem{Kors} A.\,R.\,Korselt, <<Probl{\'e}me chinois>>, L'interm{\'e}diaire des math{\'e}maticiens, vol. 6 (1899), 143.
\bibitem{Pom} C.\,Pomerance, <<Popular values of Euler's function>>, Mathematika 27 (1980), 84-89.
\end{thebibliography}
\end{document}